\documentclass[11pt,letterpaper]{article}
\pdfoutput=1

\usepackage{microtype}

\usepackage[T1]{fontenc}
\usepackage{amsthm}
\usepackage{amssymb}
\usepackage[english]{babel}
\usepackage[utf8]{inputenc}
\usepackage{amsmath}
\usepackage{amsfonts}
\usepackage{graphicx}
\usepackage[colorinlistoftodos]{todonotes}
\usepackage{mathtools}
\usepackage{enumitem}
\usepackage{thmtools}
\usepackage{fullpage}
\usepackage{booktabs}
\usepackage{bbm}
\usepackage[normalem]{ulem} 
\usepackage{hyperref}

\newcommand{\bigo}{\mathcal{O}}
\newcommand{\target}{v^*}
\newcommand{\dist}[2]{d_G(#1,#2)}
\newcommand{\st}{\hspace{0.1cm}\bigl|\bigr.\hspace{0.1cm}}
\newcommand{\eqn}[1]{\textup{eq}(#1)}
\newcommand{\pqn}[1]{\textup{pq}(#1)}
\newcommand{\Cl}{\mathcal{C}}
\newcommand{\N}{\mathbb{N}}
\newcommand{\Nat}{\mathbb{N}}
\newcommand{\SetCover}{\textsc{$3$-EXACT SET COVER}}
\newcommand{\NP}{\mathcal{NP}}
\newcommand{\G}{{\mathcal G}}

\newcommand{\Gnp}{\G(n,p)}
\newcommand{\R}{\mathbb{R}}
\newcommand{\E}{\mathbb E}
\newcommand{\Prob}{\mathbb{P}}
\newcommand{\eps}{\varepsilon}
\newcommand{\om}{\hat{\omega}}

\theoremstyle{plain}
\newtheorem{definition}{Definition}[section]
\newtheorem{lemma}[definition]{Lemma}
\newtheorem{theorem}[definition]{Theorem}

\newtheorem{observation}[definition]{Observation}
\newtheorem{corollary}[definition]{Corollary}

\title{Edge and Pair Queries---Random Graphs and Complexity}

\author{
Dariusz Dereniowski\thanks{Faculty of Electronics, Telecommunications and Informatics, Gda\'{n}sk University of Technology, Poland, email: \texttt{deren@eti.pg.edu.pl}. Author partially supported by National Science Centre, Poland, grant number 2018/31/B/ST6/00820.}
\and
Przemysław Gordinowicz\thanks{Institute of Mathematics, Lodz University of Technology Lódź, Poland, email: \texttt{pgordin@p.lodz.pl}}
\and
Paweł Prałat\thanks{Department of Mathematics, Toronto Metropolitan University, Toronto, ON, Canada, email: \texttt{pralat@torontomu.ca}}
}

\begin{document}

\maketitle

\begin{abstract}
We investigate two types of query games played on a graph, pair queries and edge queries. We concentrate on investigating the two associated graph parameters for binomial random graphs, and showing that determining any of the two parameters is $\NP$-hard for bounded degree graphs.
\end{abstract}

\section{Introduction}

Consider a query game played on a graph whose goal is to locate a vertex $\target$ that is unknown to an adaptive query algorithm.
Each query points to a pair of vertices $u$ and $v$, and the reply provides an answer that indicates which of those vertices is closer to $\target$, breaking ties arbitrarily.
The aim is to construct an algorithm performing as few queries as possible in the worst case.
In this work we consider two types of queries: \emph{edge queries} in which $u$ and $v$ need to be adjacent and \emph{pair queries} in which there is no restriction on the choice of $u$ and $v$.
The latter have not been considered in the literature before while the former have been extensively studied but only for trees.
We note that both models generalize the classical binary search on a path to arbitrary graphs.

\medskip

The game is formally defined in Subsection~\ref{sec:problem}. Our results are summarized in Subsection~\ref{sec:results} and previously known results are highlighted in Subsection~\ref{sec:related_work}. We provide three types of results. First, we reformulate the problem of finding an invisible target in the language of combinatorial game theory in which two perfect players play the game having complete information (Subsection~\ref{sec:reformulation}). We also show some universal bounds and provide a construction that shows the two graph parameters related to edge and, respectively, pair queries can be drastically different (Subsection~\ref{sec:basic_properties}). Then, we investigate the behaviour of the two parameters in binomial random graphs $\Gnp$ (Section~\ref{sec:random_graphs}). For a wide range of the parameter $p=p(n)$, the two graph parameters are predictable with high probability and turn out to be of the same order. Finally, we show that determining any of the two parameters is $\NP$-hard for bounded degree graphs (Section~\ref{sec:hardness}).

\subsection{Problem Statement}\label{sec:problem}

For an arbitrary simple graph $G$, we consider two types of \emph{query games}, which differ by the types of allowed queries.
We first define the dynamics of the game and then we define the queries.
The game is played between an (adaptive) \emph{algorithm} and an \emph{adversary}.
The adversary picks in advance a vertex $\target$ of $G$, that we call the \emph{target}.
The target is unknown to the algorithm and its goal is to locate $\target$ by asking as few queries as possible.
In each step of the game the algorithm performs a query by selecting two vertices $u$ and $v$ as outlined below depending on the type of the query.
Then, the adversary reveals which of those vertices is closer to $\target$, providing any of them when their distance to the target is the same.
The \emph{distance} $d_G(u,v)$ between any two vertices $u$ and $v$ is measured as the length of any shortest path between $u$ and $v$ in $G$.
The vertex pointed by the adversary is called the \emph{reply} to the query.
Note that the algorithm is adaptive in the sense that it is making its next query based on the graph and the replies to the previous queries.

In this work we consider two types of graph queries for the game played on $G$.
The more general one, called a \emph{pair query}, selects two arbitrary vertices $u$ and $v$ of $G$.
The second one, called an \emph{edge query}, selects two adjacent vertices $u$ and $v$.
Let $\pqn{G}$ be the least integer $t$ such that the algorithm can locate the target in at most $t$ pair queries, regardless of the strategy of the adversary. The graph parameter $\eqn{G}$ is defined similarly with the only difference that the algorithm has to perform edge queries instead. Since the latter version puts more restriction on the algorithm, for any graph $G$ we have $\pqn{G} \le \eqn{G}$.

The above definition of the process might suggest that this is an incomplete information game in which the target vertex is selected by one of the players (the adversary) but is hidden from the other player (the algorithm). However, because we concentrate on the worst-case scenario (that is, we search for the least number of queries that guarantees success, regardless where the target is and regardless of the strategy of the adversary), we can alternatively think about this game as a combinatorial game in which both players have perfect information. In particular, it implies that the graph parameters $\pqn{G}$ and $\eqn{G}$ are well-defined. We reformulate the process in this language in Subsection~\ref{sec:reformulation}.

\subsection{Notation}

In this paper, we present results obtained for the \emph{binomial random graph} $\Gnp$. More precisely, $\Gnp$ is a distribution over the class of graphs with vertex set $[n]:=\{1,\ldots,n\}$ in which every pair $\{i,j\} \in \binom{[n]}{2}$ appears independently as an edge in $G$ with probability~$p$. Note that $p=p(n)$ may (and usually does) tend to zero as $n$ tends to infinity. We say that $\Gnp$ has some property \emph{asymptotically almost surely} or a.a.s.\ if the probability that $\Gnp$ has this property tends to $1$ as $n$ goes to infinity. For more about this model see, for example,~\cite{Bollobas,JLR,Karonski_Frieze}.

Given two functions $f=f(n)$ and $g=g(n)$, we will write $f(n)=\bigo(g(n))$ if there exists an absolute constant $c \in \R_+$ such that $|f(n)| \leq c|g(n)|$ for all $n$, $f(n)=\Omega(g(n))$ if $g(n)=\bigo(f(n))$, $f(n)=\Theta(g(n))$ if $f(n)=\bigo(g(n))$ and $f(n)=\Omega(g(n))$, and we write $f(n)=o(g(n))$ or $f(n) \ll g(n)$ if $\lim_{n\to\infty} f(n)/g(n)=0$. In addition, we write $f(n) \gg g(n)$ if $g(n)=o(f(n))$ and we write $f(n) \sim g(n)$ if $f(n)=(1+o(1))g(n)$, that is, $\lim_{n\to\infty} f(n)/g(n)=1$.

We will use $\log n$ to denote a natural logarithm of $n$. For a given $n \in \N := \{1, 2, \ldots \}$, we will use $[n]$ to denote the set consisting of the first $n$ natural numbers, that is, $[n] := \{1, 2, \ldots, n\}$. Finally, as typical in the field of random graphs, for expressions that clearly have to be an integer, we round up or down but do not specify which: the choice of which does not affect the argument.

\subsection{Our Results}\label{sec:results}

Let us start with presenting results for binomial random graphs. The behaviour of $\pqn{G}$ and $\eqn{G}$ change drastically when the random graph changes its diameter. In order to control when this happens, we will use the following well-known result. 

\begin{lemma}[\cite{Bollobas}, Corollary 10.12]\label{lem:diameter}
Suppose that $d = pn \gg \log n$ and for some $i = i(n) \in \N$,
$$
d^{i+1}/n - 2 \log n \to \infty \text{ \ \ \ \ and \ \ \ \ } d^{i}/n - 2 \log n \to -\infty.
$$
Then the diameter of $\Gnp$ is equal to $i+1$ a.a.s.
\end{lemma}

This result was proved in~\cite[Theorem~6]{bollobas1981diameter} for graphs with $d \gg \log^3 n$ but in~\cite[Corollary~10.12]{Bollobas} the condition was relaxed and it is now required only that $d \gg \log n$. 

\medskip

In order to state the results in a simple form, let us concentrate on relatively dense graphs (that is, $d > n^{\eps}$ for some $\eps > 0$) and values of $d$ for which we are relatively far away from the places when the diameter changes its value. 

\begin{corollary} \label{cor:xi}
Suppose that $d = pn = n^{\xi + o(1)}$, where $\xi \in ( \frac {1}{i+1}, \frac {1}{i} )$ for some $i \in \N$ ($i$ is an arbitrarily large but fixed constant).
Then, the following properties hold a.a.s.\ for $G \in \Gnp$.
$$
\pqn{G} = \Theta \Big( \eqn{G} \Big) = \Theta \left( \frac {n \log n}{d^i} \right). 
$$
\end{corollary}

Note that for any $\xi \in ( \frac {1}{i+1}, \frac {1}{i} )$ there exists $\eps = \eps(\xi) > 0$ such that $d^{i}/n = n^{i \xi - 1 + o(1)} \le n^{-\eps} = o(1) \ll \log n$ whereas $d^{i+1}/n = n^{(i+1) \xi - 1 + o(1)} \ge n^{\eps} \gg \log n$. So, by Lemma~\ref{lem:diameter}, the diameter of $\Gnp$ is equal to $i+1$ a.a.s. More importantly, it implies that the assumptions stated in Theorems~\ref{thm:lower_bound}, \ref{thm:upper_bound_pqn}, and~\ref{thm:upper_bound_eqn} are satisfied. The first of them yields the desired lower bound for $\pqn{G}$ whereas the last two yield upper bounds for $\pqn{G}$ and, respectively, $\eqn{G}$. The corollary implies that a.a.s.\ both $\pqn{G}$ and $\eqn{G}$ are equal to $n^{f(\xi)+o(1)}$, where the function $f(\xi)$ is depicted in Figure~\ref{fig:f}.
\begin{figure}[htb]
	\begin{center}
		\includegraphics[scale=1]{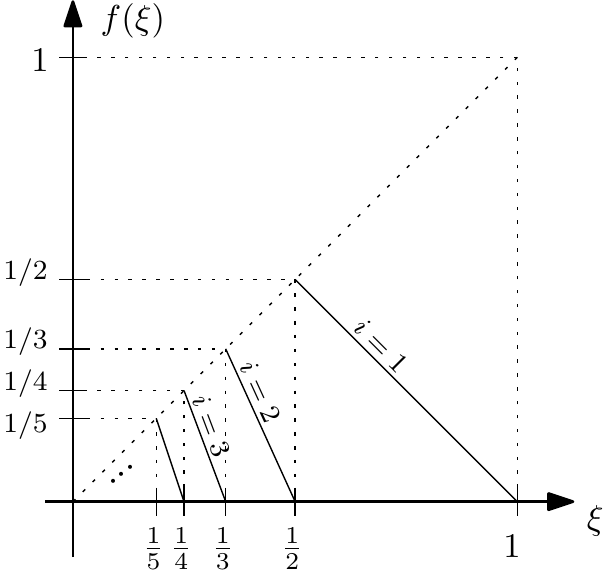}
	\end{center}
	\caption{An illustration for Corollary~\ref{cor:xi}: the function $f(\xi)$ }
	\label{fig:f}
\end{figure}

The above corollary ignores the cases when $\xi = 1/i$ for some constant $i \in \N$ and very sparse graphs (that is, when $d = n^{o(1)}$). The three theorems mentioned above provide some useful bounds for these regimes but not everything is known. For example, to keep the argument relatively short, we did not consider the cases when $i = i(n)$ is the largest natural number such that $d^i = o(n)$ but also $d^{i+1}/n - 2\log n \to -\infty$ (say, $d^{i+1} = cn$ for some constant $c >0$). Having said that, we provided a short argument dealing with a rather simple case when $p \in (0,1)$ is a fixed constant. The next corollary follows immediately from Theorem~\ref{thm:dense_graphs} which states an explicit upper bound and Observation~\ref{obs:lower} which provides a trivial lower bound but of matching order.

\begin{corollary}
Fix $p \in (0,1)$. 
The following property holds a.a.s.\ for $G = (V,E) \in \Gnp$:
$$
\pqn{G} = \Theta \Big( \eqn{G} \Big) = \Theta \left( \log n \right). 
$$
\end{corollary}

\bigskip

In terms of complexity, we prove the following result.

\begin{theorem} \label{thm:NP}
Given a graph of diameter at most $3$ and an integer $\ell$, deciding whether $\pqn{G}\leq \ell$ or $\eqn{G}\leq \ell$ is $\NP$-complete.
\end{theorem}

\subsection{Related Work}\label{sec:related_work}

The edge query model has been studied for trees and paths as a special case.
This is due to the fact that the problem for paths is the classical binary search, where most of the focus has been on noisy comparison models, see e.g.\ \cite{Ben-OrH08,DaganFGM17,DereniowskiLU21,DereniowskiTUW19,FeigeRPU94,KarpK07}.
As for the case of trees, they generalize binary search to searching partial orders with some interesting applications e.g.\ in automated program testing~\cite{Ben-AsherFN99}.
It turns out that $\eqn{T}$ for a tree $T$ can be computed in linear time~\cite{MozesOW08}.
Interestingly, this problem re-appeared under various different names like minimum height elimination trees~\cite{Liu90}, LIFO-search~\cite{GiannopoulouHT12}, tree-depth~\cite{NesetrilM06}, or edge ranking.
Due to the latter connection, a linear-time algorithm has been independently shown for trees in~\cite{LamY01}. Edge queries in trees are also of interest in applications like parallel database query processing, where one wants to find a spanning tree $T$ of the graph representing a database query that minimizes $\eqn{T}$~\cite{MakinoUI01}.
For some research on random partial orders in this context see~\cite{CarmoDKL04}.

It is natural to further generalize edge queries to the weighted variant, where the weight of an edge represents the duration of a query~\cite{LaberMP02}.
Such an extension turns out to be $\NP$-complete for some quite narrow subclasses of trees~\cite{CicaleseJLV12,CicaleseKLPV16,Dereniowski06}.
On the other hand, there has been a series of results on approximation algorithms~\cite{CicaleseJLV12,CicaleseKLPV16,Dereniowski06} with the best approximation ratio to date being $\bigo(\sqrt{\log n})$ that was proved in~\cite{DereniowskiKUZ17}.
Finally, let us mention another generalization of binary search in linear orders to graphs~\cite{Zadeh16,DereniowskiTUW19} with some interesting applications in machine learning \cite{Zadeh17}.

\medskip

Another related graph parameter is the the centroidal dimension of a graph~\cite{foucaud2014centroidal}, which may be viewed as a variant of the well known metric dimension~\cite{melter1976metric,slater1975leaves} that provides less information. Indeed, the idea is similar to the pair query model but the goal is to localize the hidden target by just a single query (usually involving more than two vertices), minimizing the number of vertices to probe. For the metric dimension, the reply returns a vector of distances between the target and each of the examined vertices. On the other hand, for the centroidal variant, the reply informs the player when the distance is zero; otherwise, for every pair of examined vertices the result of comparison of distances is returned (that is, which vertex is closer or whether distances are equal). Game theoretical variants of the above parameters have recently been intensively studied: the localization game~\cite{bonato2020bounds,bosek2018localization,carraher2012locating,seager2014locating} and centroidal localization game~\cite{bosek2018centroidal}. During the game, more queries are considered but the target can move to a neighbouring vertex between queries. The main goal is typically to minimize the number of vertices examined in each round but the number of rounds may also be minimized~\cite{BEHAGUE202280}. 

\medskip

Although the query games have not yet been studied for random graphs, there are some known results for the closely related graph parameters mentioned above.  An asymptotic behaviour of  the metric dimension of dense binomial random graphs $\Gnp$ is obtained in~\cite{bollobas2013metric}, while for sparse, subcritical $\Gnp$ and uniform random trees and forests, its asymptotic distribution is shown in~\cite{mitsche2015limiting}. The localization number for dense binomial random graphs (in particular, in the regime in which $\Gnp$ has diameter two a.a.s.) was studied in~\cite{dudek2019note}. The bounds for dense graphs were consecutively improved in~\cite{dudek2022localization}, and the arguments were extended to sparse graphs. The localization game was also recently studied for random geometric graphs~\cite{lichev2021localization}. We direct the reader to the book~\cite{bonato2017graph} for more graph searching games in the context of random graphs.

\section{Preliminaries}

\subsection{Reformulation of the Process}\label{sec:reformulation}

In this subsection, we provide a reformulation of the game that will be easier to deal with when proving our results. In particular, this reformulation will show that the process can be viewed as a combinatorial game. Indeed, it is convenient to think about this game as the target is \emph{not} selected a priori by the adversary but rather as a process in which after a $t$-th query the set $V_t$ of \emph{potential targets} is maintained by both players.

Let us denote by $\dist{u}{v}$ the distance between $u$ and $v$ in a given graph $G=(V,E)$.
For any pair of vertices $u$ and $v$, let
$$\Cl(u,v)=\{ x \in V \st \dist{u}{x} \leq \dist{v}{x} \},$$
that is, $\Cl(u,v)$ is the set of vertices that are closer to $u$ than to $v$, or at the same distance from both $u$ and $v$.
For a reply $u$ to the query on $u,v$, we say that a vertex $x$ is \emph{compatible} with the reply if $x\in \Cl(u,v)$.

With these definitions at hand, we may formally define $V_t$ to be the set of vertices of $G$ that are compatible with each of the first $t$ replies. In other words, each vertex from the set $V_t$ satisfies all inequalities associated with the first $t$ queries and the corresponding replies. On the other hand, no vertex from $V \setminus V_t$ has this property.
The game starts with $V_0 = V$. At the beginning of round $t \in \Nat$, the first player (the algorithm) presents a pair of vertices $(u_t, v_t)$ which partitions $V_{t-1}$, the set of potential targets from the previous round, into three sets: 
\begin{itemize}
    \item $V_t^< := V_{t-1} \cap (\Cl(u_t,v_t) \setminus \Cl(v_t,u_t)$), the set of vertices of $V_{t-1}$ that are closer to $u_t$ than to $v_t$,
    \item $V_t^> := V_{t-1} \cap (\Cl(v_t,u_t) \setminus \Cl(u_t,v_t)$), the set of vertices of $V_{t-1}$ that are closer to $v_t$ than to $u_t$,
    \item $V_t^= := V_{t-1} \cap (\Cl(u_t,v_t) \cap \Cl(v_t,u_t)$), the set of vertices of $V_{t-1}$ that are at the same distance from both $u_t$ and $v_t$.
\end{itemize}
The second player (the adversary) has now two options: fix $V_t = V_{t-1} \cap \Cl(u_t,v_t) = V_t^< \cup V_t^=$ (that is, reply with vertex $u_t$) or fix $V_t = V_{t-1} \cap \Cl(v_t,u_t) = V_t^> \cup V_t^=$ (that is, reply with vertex $v_t$). After making the decision, the round is over and we move to the next round.
The game ends when $|V_t|=1$ since the only element in $V_t$ must be the target.
With this point of view, one may think about the adversary as also being adaptive in the sense that the goal of the adversary is to provide replies which result in maximizing the value of $t$ for which $V_t$ is a singleton. 

\subsection{Some Basic Properties}\label{sec:basic_properties}

In order to warm-up, let us start with a few simple observations highlighting some properties of the two graph parameters we are concerned with in this paper.

\begin{observation} \label{obs:lower}
For any connected graph $G$ on $n$ vertices, 
$$
\log_2 n \le \pqn{G} \le \eqn{G} \le n-1.
$$
In fact, there exists a strategy of the algorithm that in each round eliminates at least one vertex from the search space. (This, of course, implies an upper bound of $n-1$.)
\end{observation}

\begin{proof}
The inequality $\pqn{G} \le \eqn{G}$ follows immediately from the fact that pair queries are more general than edge queries. Hence, any strategy asking edge queries can be used by more flexible algorithm asking pair queries.

The lower bound of $\log_2 n$ is due to the fact that there exists a strategy (simple and greedy) for the adversary that guarantees that each pair query reduces the search space by at most half. Indeed, recall that in round $t \in \Nat$, set $V_{t-1}$ is partitioned into three sets, $V_t^<$, $V_t^>$, and $V_t^=$ (the partition depends on the algorithm) and then $V_t = V_t^< \cup V_t^=$ or $V_t = V_t^> \cup V_t^=$ (the choice belongs to the adversary). Hence, indeed, there exists a strategy for the adversary that guarantees that for each $t \in \Nat$ we have $|V_t| \ge |V_{t-1}| / 2$. Hence, $|V_t| \ge |V_0| / 2^t = n / 2^t$ implying that $|V_t| > 1$ if $t < \log_2 n$. The lower bound holds.

To prove the upper bound of $n-1$, let us present another simple strategy for the algorithm, this time performing edge queries, that guarantees that in each round at least one vertex is eliminated from the search space. Assume that $|V_{t-1}| \ge 2$, and let $x$, $y$, $x \neq y$, be any two vertices from $V_{t-1}$. Let $(x=z_0, z_1, \ldots z_k = y)$ be any shortest path between $x$ and $y$. The algorithm, for example, may select vertices $u_t = x = z_0$ and $v_t = z_1$ as its edge query to eliminate $x$ or $y$ from the search space, giving $|V_t| \le |V_{t-1}| - 1$. This finishes the proof of the observation.
\end{proof}

The next observation shows that the two easy bounds we proved above are sharp. We leave it as an easy exercise for the reader and only provide a hint that edge queries can mimic the binary search on a path.

\begin{observation} \label{obs:simple}
For any $n \in \N \setminus \{1\}$, 
$\eqn{K_n} = \pqn{K_n} = n-1$, 
$\eqn{K_{1,n-1}} = \pqn{K_{1,n-1}} = n-1$, 
$\eqn{P_n} = \pqn{P_n} = \lceil\log_2n\rceil$.
\end{observation}

The last observation shows that the difference between the two variants of the game may be asymptotically as large as possible.

\begin{observation}
For any $n \in \N \setminus \{1\}$, there exists a graph $G$ on $\Theta(n)$ vertices such that $\eqn{G} = \Omega(n)$ and $\pqn{G} = \bigo(\log n)$.
\end{observation}
\begin{proof}
For a fixed $n \in \N \setminus \{1\}$, let $k=k(n)=\lfloor\log_2n\rfloor$. Let $T_k$ be a full binary tree of height $k$. (The height is measured as the length of a path from the root to any of the leaves.)
We say that the vertices at distance $i$ from the root in $T_k$ are at \emph{level} $i$, $i\in\{0,\ldots,k\}$. We denote the set of vertices at level $i$ by $L_i$; in particular, $L_k$ is the set of the leaves of $T_k$ and $L_0$ consists of only one vertex, the root of $T_k$. It will be convenient to use $T_k(x)$, $x\in V(G)$, to denote the subtree of $T_k$ induced by the set consisting of $x$ and all of its descendants in $T_k$. 

We construct a graph $G = G_k$ on the same vertex set as $T_k$ (that is, $V(G_k) = V(T_k)$). 
Note that $|V(G_k)| = 2^{k+1} - 1$, hence $n \le |V(G_k)| < 2n$.
The edge set $E(G_k)$ is defined as follows. First of all, we join any two leaves of $T_k$ by an edge, that is, the set $L_k$ induces a complete graph in $G_k$. Then, for any non-leaf $u$ of $T_k$ (that is, $u \in V(T_k) \setminus L_k$) we join $u$ to all the leaves of the subtree $T_k(u)$ of $T_k$ rooted at $u$. 
Note that not all edges of tree $T_k$ are present in $G_k$---see Figure~\ref{fig:obsdif} for an example of this construction. 

\begin{figure}[htb]
	\begin{center}
		\includegraphics[width = 0.5\textwidth]{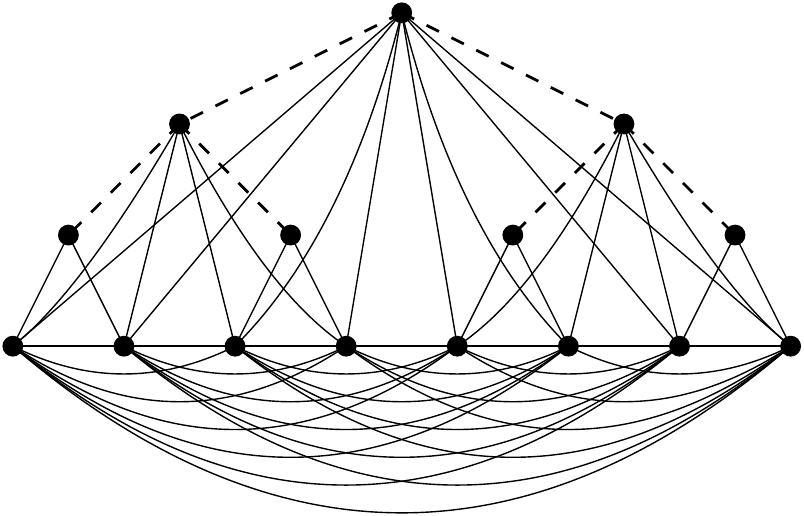}
	\end{center}
	\caption{The graph $G_3$. The dashed edges are edges of $T_3$ that are \emph{not} present in $G_3$. }\label{fig:obsdif}
\end{figure}

\medskip

Let us start with providing a strategy for the algorithm that proves that $\pqn{G_k} \le 2k$, hence $\pqn{G_k}=\bigo(\log n)$.
We will distinguish two phases. During the first phase the algorithm performs $k$ pair queries starting from vertices that are close to the root and moving toward the leaves, each time essentially removing one branch of $T_k$ from the search space.

Formally, as always, the algorithm starts with $V_0 = V(G_k) = V(T_k)$ and $u_0$ being the root of $T_k$.
In the first round, the algorithm queries two children of $u_0$ that we denote by $v^1_0$ and $v^2_0$. By symmetry of graph $G_k$, without loss of generality we may assume that the reply is $v^1_0$. Note that all vertices in $T_k(v^1_0)$ are closer to $v^1_0$ than to $v^2_0$. The only vertex that is at the same distance from both $v^1_0$ and $v^2_0$ is the root of $T_k$ (that is, $V_1^= = \{u_0\}$). Hence, the algorithm moves to the next round with $V_1 = V_0 \cap \Cl(v^1_0, v^2_0) = V_0 \setminus V(T_k(v^2_0))$ and we prepare ourselves for the next round by fixing $u_1 = v^1_0$. Note that $V_1 = V(T_k(u_1)) \cup \{u_0\}$, where $\{u_0\}$ is the set of ancestors of $u_1$ in $T_k$.
We recursively follow the same strategy as for the first round. Suppose that at the end of round $t$, for some $1 \le t < k$, we have a vertex $u_t$ at level $t$ identified such that $V_t = V(T_k(u_t)) \cup \{u_0, \dots u_{t-1}\}$, where $u_i$ is the ancestor of $u_t$ at level $i < t$. At round $t+1$, the algorithm queries two sons of $u_t$, $v^1_t$ and $v^2_t$. As before, because of symmetry, without loss of generality we may assume that the reply is $v^1_t$. Then, we fix $u_{t+1} = v^1_t$ such that 
$$
V_{t+1} = V_t \cap \Cl(v^1_t, v^2_t) = V_t \setminus V(T_k(v^2_t)) = V(T_k(u_{t+1})) \cup \{u_0, \dots u_{t}\}.
$$   
It follows that after $k$ pair queries, regardless of the strategy of the adversary, the algorithm learns that the target is hidden in the set $V_k = \{u_0, \dots, u_k\}$ where $u_k$ is some leaf of $T_k$, while $u_i$ for $i < k$ is the ancestor of $u_k$ at level $i$. 

During the second phase, by Observation~\ref{obs:lower}, the algorithm may eliminate at least one vertex at a time. The target is then identified in at most $k$ additional rounds and so the total number of pair queries is at most $2k=\bigo(\log n)$.

\medskip

We now move to the proof of the fact that $\eqn{G_k}=\Omega(n)$.
In order to simplify the argument, we assume that the algorithm is provided with some additional information, namely, that the target is in $L_k$. (In other words, we start the game with $V_0 = L_k$ instead of $V_0 = V(G_k)$.) We will show that even with this additional information, the algorithm cannot quickly identify the target.

Consider a query to an edge $uu' \in E(G_k)$. Since each edge of $G_k$ has at least one leaf from $T_k$, without loss of generality we may assume that $u' \in L_k$. If $u \not \in L_k$, then the adversary might reply with $u'$ which provides no new information to the algorithm.
Thus, the algorithm is forced to query edges within the complete graph induced by $L_k$. However, since all vertices in $L_k$ are at distance one from each other, each such query eliminates at most one vertex from the search space. It follows that the number of edge queries that is needed to identify the target is at least $|L_k|-1 = 2^k-1 \ge \frac{n}{2}$. 
This finishes the proof of the lower bound of $\eqn{G_k}$ and so also the proof of the observation. 
\end{proof}

\section{Random Graphs}\label{sec:random_graphs}

Let us first state a specific instance of Chernoff's bound that we will find often useful. Let $X \in \textrm{Bin}(n,p)$ be a random variable with the binomial distribution with parameters $n$ and $p$. Then, a consequence of Chernoff's bound (see e.g.~\cite[Corollary~2.3]{JLR}) is that 
\begin{equation}\label{chern}
\Prob \Big( |X-\E [X]| \ge \eps \ \E [X] \Big) \le 2\exp \left( - \frac {\eps^2 \ \E [X]}{3} \right)  
\end{equation}
for  $0 < \eps < 3/2$. However, at some point we will need need a stochastic upper bound for $X$ when $\E[X]$ is small. In such situations the following bound can be applied instead of~(\ref{chern}) (see e.g.~\cite[Theorem~2.1]{JLR}):
\begin{equation}\label{chern2}
\Prob( X \ge \E [X] + u ) \le \exp \left( - \frac {u^2}{2(\E [X] + u/3)} \right).
\end{equation}

\subsection{Lower Bound}

Let us start with the following expansion properties of random graphs. For a vertex $v \in V$, let $N_i(v)$ and $N_{\le i}(v) = \bigcup_{j=0}^i N_i(v)$ denote the set of vertices at distance $i$ from $v$ and the set of vertices at distance at most $i$ from $v$, respectively. For any $Z \subseteq V$, let $N_i(Z) = \bigcup_{v \in Z} N_i(v)$ and $N_{\le i}(Z) = \bigcup_{v \in Z} N_{\le i}(v)$. 

\begin{lemma}\label{lem:gnp exp}
Suppose that $\omega \log n \le d=pn \le n / \omega = o(n)$, where $\omega=\omega(n)$ is any function tending to infinity as $n \to \infty$ such that $\omega \le \log n$.  

Then, the following properties hold a.a.s.\ for $G = (V,E) \in \Gnp$.
\begin{itemize}
\item [(a)] Suppose that for some $i = i(n) \in \N$ and $Z \subseteq V$ with $z=|Z|$ we have $zd^i = o(n)$. Then,
$$
\left| N_{\le i}(Z) \right| = \left(1+ \bigo \left( \frac {1}{\sqrt{\omega}}\right) + \bigo \left( \frac{zd^i}{n} \right) \right) d^i |Z| \sim d^i |Z|,
$$
and so also
$$
\left| N_i(Z) \right| = \left(1+ \bigo \left( \frac {1}{\sqrt{\omega}}\right) + \bigo \left( \frac{zd^i}{n} \right) \right) d^i |Z| \sim d^i |Z|.
$$
\item [(b)] Let $i=i(n) \in \N$ be the largest integer such that $d^i = o(n)$. Let 
$$
z := \frac {n}{d^i} \log (d^i) \left(1 - \frac {1}{\omega^{1/3}} \right) = \Theta \left( \frac {n}{d^i} \log n \right).
$$
Then, for any $Z \subseteq V$ with $|Z|=z$ there are at least two vertices at distance at least $i+1$ from $Z$, that is, $|N_{\le i}(Z)| \le n-2$.
\end{itemize}
\end{lemma}

\begin{proof}
In order to investigate the expansion property of neighbourhoods, let $Z \subseteq V$, $z=|Z|$, and consider the random variable $X = X(Z) = |N_{\le 1}(Z)|$. We will bound $X$ in a stochastic sense. There are two things that need to be estimated: the expected value of $X$, and the concentration of $X$ around its expectation.

Since for $x=o(1)$ we have $(1-x)^z = e^{-xz(1+\bigo(x))}$ and also $e^{-x}=1-x+\bigo(x^2)$, it is clear that
\begin{eqnarray}
\E [X] &=& n - \left(1- \frac {d}{n-1} \right)^z (n-z) \nonumber \\
&=& n - \exp \left( - \frac {dz}{n} (1+\bigo(d/n)) \right) (n-z) \nonumber \\
&=& dz (1+\bigo(dz/n) ), \label{eq:expX}
\end{eqnarray}
provided $dz = o(n)$. It follows from Chernoff's bound~\eqref{chern} applied with $\eps = 3/{\sqrt{\omega}}$ that the expected number of sets $Z$ satisfying  
$$
\big| |N_{\le 1}(Z)| - \E[|N_{\le 1}(Z)|] \big| > \eps d|Z|
$$ 
is at most 
$$
2 \binom{n}{z} \exp \left( - \frac {\eps^2 d z}{3+o(1)} \right) \le 2 n^z \exp \left( - \frac {\eps^2 (\omega \log n) z }{3+o(1)} \right) = 2 n^z \exp \Big( - (3+o(1)) z \log n \Big) = o(1),
$$
since $d \geq \omega \log n$. Hence the statement holds for $i=1$ a.a.s. 

Now, we will estimate the cardinalities of $N_{\le i}(Z)$ up to the $i$'th iterated neighbourhood, provided $zd^i = o(n)$ and thus $i = O(\log n /\log \log n)$. Suppose then that for some $i \ge 2$, $Y = N_{\le i-1}(Z)$ satisfies $|Y| \sim d^{i-1} z$ and $|Y| = o(n/d)$. It follows from~(\ref{eq:expX}) and~(\ref{chern}) (this time applied with $\eps = 4 d^{-(i-1)/2} \omega^{-1/2}$) that with probability at least $1-\beta$,
$$
|N_{\le 1}(Y)| = d |Y| \left( 1+ \bigo \left( d|Y|/n \right) + \bigo \left( d^{-(i-1)/2} \omega^{-1/2} \right) \right),
$$
where the bounds in $\bigo()$ notation are uniform. The failure probability is at most 
$$
\beta = 2 \exp \left( - \frac {\eps^2 d |Y|}{3+o(1)} \right) = 2 \exp \left( - \frac {\eps^2 d^i z}{3+o(1)} \right) \le \exp \left( - 5 d z / \omega \right) \le \exp( - 5 z \log n) \le n^{-z} n^{-4}.
$$
As we want a result that holds a.a.s., we may assume this statement holds deterministically, since there are only $\bigo(n^z \log n)$ choices for $Z$ and $i$. 

Given this assumption, we have good bounds on the ratios of the cardinalities of $N_{\le 1}(Z)$, $N_{\le 1}(N_{\le 1}(Z)) = N_{\le 2}(Z)$, and so on. Since $i=O(\log n / \log \log n)$ and $\sqrt{\omega} \le (\log n)^{1/2}$, the cumulative multiplicative error term is
\begin{align*}
(1+&\bigo(dz/n) + \bigo(1/\sqrt{\omega})) \prod_{j=2}^i \left( 1+ \bigo \left( d^jz/n \right) + \bigo \left( \omega^{-1/2} d^{-(j-1)/2} \right) \right) \\
&= (1+\bigo(1/\sqrt{\omega}) + \bigo(d^iz/n) )  \prod_{j=5}^{i-2} \left( 1+ \bigo \left( \log^{-2} n \right) \right) = (1+ \bigo(1/\sqrt{\omega}) + \bigo(d^iz/n) ),
\end{align*}
and the proof of part (a) is complete.

\bigskip

In order to prove part (b), let us fix any set of vertices $Z \subseteq V$ of size 
$$
|Z|=z=\frac {n}{d^i} \log (d^i) \left(1 - \frac {1}{\omega^{1/3}} \right),
$$ 
where $i=i(n) \in \N$ is the largest integer such that $d^i = o(n)$. Note that, by the definition of $i$, 
$$
\left( \frac 12 + o(1) \right) \frac {n \log n}{d^i} \le z \le \frac {n \log n}{d^i}.
$$
Since $zd^{i-1}/n = \bigo( \log n / d ) = \bigo( \omega^{-1} )$, we may expose $N_{\le i-1}(Z)$ and based on part~(a) we may assume that 
$$
|N_{\le i-1}(Z)| = (1+\bigo(\omega^{-1/2})) \, zd^{i-1}.
$$ 
There are $(1+\bigo( \omega^{-1} )) \, n$ vertices outside of $N_{\le i-1}(Z)$. The probability that a given one of them is not adjacent to any vertex of $N_{\le i-1}(Z)$ is $(1-p)^{(1+\bigo(\omega^{-1/2})) \, zd^{i-1}}$ and so the probability that all of them belong to $N_{\le i}(Z)$ (that is, $N_{i+1}(Z) = \emptyset$) is equal to
\begin{eqnarray*}
q &=& \left( 1 - (1-p)^{(1+\bigo(\omega^{-1/2})) \, zd^{i-1}} \right)^{(1+\bigo( \omega^{-1} )) \, n} \\
&=& \left( 1 - \exp \left( - \frac {(1+\bigo(\omega^{-1/2})) \, zd^{i}}{n} \right) \right)^{(1+\bigo( \omega^{-1} )) \, n} \\
&=& \left( 1 - \exp \left( - \left( 1 - \frac {1+o(1)}{\omega^{1/3}} \right) \log(d^i) \right) \right)^{(1+\bigo( \omega^{-1} )) \, n}.
\end{eqnarray*}
Since $\log(d^i) = \Theta(\log n)$ and $\omega \le \log n$, $(1+o(1)) \log(d^i) / \omega^{1/3} \ge 3 \log \log n$ and so
\begin{eqnarray*}
q &\le& \left( 1 - \exp \left( - \log(d^i)  + 3 \log \log n \right) \right)^{(1+\bigo( \omega^{-1} )) \, n} \\
&=& \left( 1 - \frac {\log^3 n}{d^i} \right)^{(1+\bigo( \omega^{-1} )) \, n} \\
&=& \exp \left( - (1+o(1)) \, \frac {n \log^3 n}{d^i} \right).
\end{eqnarray*}
Similarly, the probability that all but exactly one vertex outside of $N_{\le i-1}(Z)$ belong to $N_{\le i}(Z)$ is equal to
\begin{align*}
(1+o(1)) n \, (1-p)^{(1+\bigo(\omega^{-1/2})) \, zd^{i-1}} & \left( 1 - (1-p)^{(1+\bigo(\omega^{-1/2})) \, zd^{i-1}} \right)^{(1+\bigo( \omega^{-1} )) \, n} \\
& \le nq \le \exp \left( - (1+o(1)) \, \frac {n \log^3 n}{d^i} \right).
\end{align*}
On the other hand, the number of sets of size $z$ is, trivially, at most 
$$
\binom{n}{z} \le n^z = \exp \left( z \log n \right) \le \exp \left( \frac {n \log^2 n}{d^i} \right) = o( nq ). 
$$
We get that the expected number of sets $Z \subseteq V$ with $|Z|=z$ with the property that at most one vertex is at distance at least $i+1$ from $Z$ is equal to $o(1)$ and so the desired property holds a.a.s.\ by the first moment method. This finishes part~(b) and the proof of the lemma. 
\end{proof}

Now, we are ready to prove a lower bound for $\pqn{G}$.

\begin{theorem}\label{thm:lower_bound}
Suppose that $\omega \log n \le d=pn \le n / \omega = o(n)$, where $\omega=\omega(n)$ is any function tending to infinity as $n \to \infty$ such that $\omega \le \log n$.  
Let $i=i(n) \in \N$ be the largest integer such that $d^i = o(n)$. 
Suppose that $d^{i+1} / n - 2 \log n \to \infty$.

Then, the following property holds a.a.s.\ for $G = (V,E) \in \Gnp$:
$$
\eqn{G} \ge \pqn{G} \ge k := \frac {n}{2 d^i} \log (d^i) \left(1 - \frac {1}{\omega^{1/3}} \right) = \Omega \left( \frac {n \log n}{d^i} \right).
$$
\end{theorem}
\begin{proof}
Note that we assumed that $d^{i+1} / n - 2 \log n \to \infty$ and by definition of $i$ we get that $d^{i} / n - 2 \log n \to -\infty$. Hence, by Lemma~\ref{lem:diameter} we get that the diameter of $\Gnp$ is equal to $i+1$ a.a.s.

\medskip

Typically, in order to establish lower bounds for the length of the game, one needs to show a strategy for the second player (the adversary) and prove that this strategy guarantees that the game is not over after $k$ rounds. However, our ``board'' (random graph $\Gnp$) has nice expansion properties that guarantee that the game cannot finish earlier \emph{regardless of} what the adversary does.

Suppose then that the first player (the algorithm) presents a sequence of pairs of vertices $(u_t,v_t)_{t \in [k]}$ during the first $k$ rounds of the game. Let 
$$
Z = \{ u_t : t \in [k] \} \cup \{ v_t : t \in [k] \}
$$
be the set of all vertices presented during this part of the game. Clearly, $|Z| \le 2k$. We may add some additional vertices to $Z$, if needed, so that $|Z| = 2k$. By Lemma~\ref{lem:gnp exp}(b), we may assume that there are at least two vertices (say, $x$ and $y$) that are at distance at least $i+1$ from $Z$. Since the diameter of $G$ is equal to $i+1$, these two vertices are at distance exactly $i+1$ from any vertex in $Z$ and so, in particular, from any vertex that was presented up to this point of the game. We get that $\{x,y\} \subseteq V^=_t$ for all $t \in [k]$, which implies that $x$ and $y$ are compatible with all replies so far. As a result, regardless of what the adversary does, $\{x,y\} \subseteq V_k$ and so the game is not over yet. This finishes the proof of the theorem. 
\end{proof}

\subsection{Upper Bounds}

As usual, let us start with some useful expansion properties of random graphs. 

\begin{lemma}\label{lem:gnp exp2}
Suppose that $\omega \log n \le d=pn \le n / \omega = o(n)$, where $\omega=\omega(n)$ is any function tending to infinity as $n \to \infty$ such that $\omega \le \log n$.  Let $i=i(n) \in \N$ be the largest integer such that $d^i = o(n)$. Finally, let $\om = \om(n) = \min( \sqrt{\omega}, n/d^i ) \to \infty$ as $n \to \infty$.

The following properties hold a.a.s.\ for $G = (V,E) \in \Gnp$. 
For any vertices $x,y \in V$ ($x \neq y$) we partition the set of vertices $V$ into the following 4 parts:
\begin{eqnarray*}
X &=& N_i(x) \setminus N_{\le i}(y) \\
Y &=& N_i(y) \setminus N_{\le i}(x) \\
Z &=& V \setminus N_{\le i}(\{x,y\}) \\
R &=& V \setminus (X \cup Y \cup Z) = (N_i(x) \cap N_i(y)) \cup N_{\le i-1}(\{x,y\}).  
\end{eqnarray*}
Then, the following holds:
\begin{itemize}
\item [(a)] $|X| = \left(1+ \bigo \left( \om^{-1} \right) \right) d^i \sim d^i$.
\item [(b)] $|Y| = \left(1+ \bigo \left( \om^{-1} \right) \right) d^i \sim d^i$.
\item [(c)] $|Z| = n - \left(2+ \bigo \left( \om^{-1} \right) \right) d^i \sim n$.
\item [(d)] $|R| = \bigo \left( \om^{-1} \right) d^i = o(d^i)$.
\end{itemize}
Moreover, the following holds, provided that $d^{i+1} / n \gg \om \log n$:
\begin{itemize}
\item [(e)] Each vertex in $X \cup Y \cup Z$ has $\bigo \left( \om^{-1} \right) d^{i+1} / n$ neighbours in $R$.
\item [(f)] Each vertex in $X \cup Z$ has $\left(1+ \bigo \left( \om^{-1/2} \right) \right) d^{i+1} / n \sim d^{i+1}/n$ neighbours in $Y$.
\item [(g)] Each vertex in $Y \cup Z$ has $\left(1+ \bigo \left( \om^{-1/2} \right) \right) d^{i+1} / n \sim d^{i+1}/n$ neighbours in $X$.
\end{itemize}
\end{lemma}

\begin{proof}
Parts~(a)--(d) follow immediately and deterministically from Lemma~\ref{lem:gnp exp}(a). Since we aim for a statement that holds a.a.s., we may assume that the property stated in Lemma~\ref{lem:gnp exp}(a) holds for any set $Z \subseteq V$ with $|Z| \le 2$:
\begin{eqnarray}
\left| N_i(Z) \right| &=& \left(1+ \bigo \left( \frac {1}{\sqrt{\omega}}\right) + \bigo \left( \frac{d^i}{n} \right) \right) d^i |Z| \sim d^i |Z|,\label{eq:ball_of_Z} \\
\left| N_{\le i}(Z) \right| &=& \left(1+ \bigo \left( \frac {1}{\sqrt{\omega}}\right) + \bigo \left( \frac{d^i}{n} \right) \right) d^i |Z| \sim d^i |Z|.\label{eq:sphere_of_Z}
\end{eqnarray}
Now, for a given $x,y \in V$ ($x \neq y$) we may apply~\eqref{eq:ball_of_Z}--\eqref{eq:sphere_of_Z} with $Z=\{x,y\}$, $Z=\{x\}$, and $Z=\{y\}$ to get the desired properties. This finishes the proof of parts~(a)--(d).

\medskip

In order to prove part~(e), let us concentrate on any pair of vertices $x,y \in V$ ($x \neq y$). We first expose edges up to the $(i-1)$th neighbourhood of $\{x,y\}$. Since we aim for a statement that holds a.a.s., by Lemma~\ref{lem:gnp exp}(a) we may assume that $|N_{\le i-1}(\{x,y\})| \sim 2d^{i-1} = \bigo \left( \om^{-1} \right) d^i$. For any vertex $v \notin N_{\le i-1}(\{x,y\})$, let $Q$ be the number of neighbours of $v$ in $N_{\le i-1}(\{x,y\})$. Clearly, $Q$ is a binomial random variable with expectation equal  to $|N_{\le i-1}(\{x,y\})| d/n = \bigo \left( \om^{-1} \right) d^{i+1}/n$. It follows from Chernoff's bound~\eqref{chern2} that
$$
\Prob( Q \ge \E [Q] + \om^{-1} d^{i+1}/n ) \le \exp \left( - \Omega \left( \om^{-1} d^{i+1}/n \right) \right) = o(1/n^3),
$$
since it is assumed that $d^{i+1} / n \gg \om \log n$. By the union bound over $x,y$, and $v$, we may assume that all vertices outside of $N_{\le i-1}(\{x,y\})$ have $\bigo \left( \om^{-1} \right) d^{i+1}/n$ neighbours in $N_{\le i-1}(\{x,y\})$.

Let us now expose edges up to the $i$th neighbourhood of $\{x,y\}$, which identifies sets $X$, $Y$, $Z$, and $R$.
\begin{figure}[htb]
	\begin{center}
		\includegraphics[scale=0.85]{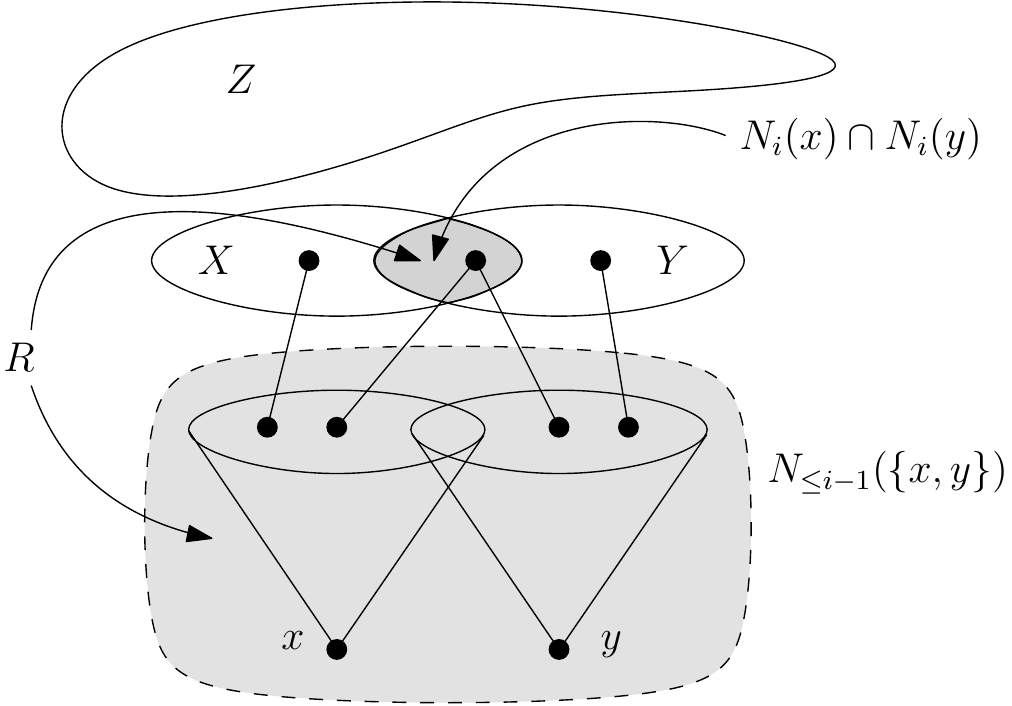}
	\end{center}
	\caption{An illustration of the partition of $V$: $X,Y,Z$, and $R = (N_i(x) \cap N_i(y)) \cup N_{\le i-1}(\{x,y\})$ (grey area).}
	\label{fig:xyz}
\end{figure}
Since we aim for a statement that holds a.a.s., we may assume that properties (a)--(d) hold. More importantly, note that edges within $V \setminus N_{\le i-1}(\{x,y\}) = X \cup Y \cup Z \cup (N_i(x) \cap N_i(y))$ are not exposed yet---see Figure~\ref{fig:xyz} for an illustration. For any $v \in X \cup Y \cup Z$, the expected number of neighbours of $v$ in $N_i(x) \cap N_i(y)$ is at most $|R| d/n = \bigo \left( \om^{-1} \right) d^{i+1}/n$. After applying Chernoff's bound~\eqref{chern2} again we get that a.a.s., for all $x,y$, all vertices in $X \cup Y \cup Z$ have $\bigo \left( \om^{-1} \right) d^{i+1} / n$ neighbours in $N_i(x) \cap N_i(y)$. This, together with the previous observation, implies the same bound for the number of neighbours in $R = (N_i(x) \cap N_i(y)) \cup N_{\le i-1}(\{x,y\})$, which concludes the proof of part~(e).

\medskip

To prove part~(f), let us expose edges up to the $i$th neighbourhood of $\{x,y\}$ and fix any $v \in X \cup Z$. Let $Q$ be the number of neighbours of $v$ in $Y$. Clearly, $\E [Q] = |Y| d/n = \left(1+ \bigo \left( \om^{-1} \right) \right) d^{i+1} / n \gg \om \log n$. The conclusion follows from Chernoff's bound~\eqref{chern} (applied with $\eps = \om^{-1/2}$), since
$$
\Prob \Big( |Q-\E [Q]| \ge \eps \ \E [Q] \Big) \le 2\exp \left( - \Omega( \eps^2 \E[Q] ) \right) = o(1/n^3).
$$ 
The proof of part~(g) is exactly the same. This finishes the proof of the lemma. 
\end{proof}

Let us start with a proof of an upper bound for $\pqn{G}$. Of course, since $\pqn{G} \le \eqn{G}$, any upper bound for $\eqn{G}$ implies the same bound for $\pqn{G}$. There are two reasons why we present them independently. First of all, the corresponding strategies for the first player (the algorithm) are slightly different but the proofs that they are winning strategies are almost identical. We present details for a slightly easier proof for $\pqn{G}$ and then we will discuss how to adjust it to deal with $\eqn{G}$. Moreover, the second proof requires slightly stronger assumptions for an asymptotic value of $d^{i+1} / n$. 

\begin{theorem}\label{thm:upper_bound_pqn}
Suppose that $\omega \log n \le d=pn \le n / \omega = o(n)$, where $\omega=\omega(n)$ is any function tending to infinity as $n \to \infty$ such that $\omega \le \log n$.  
Let $i=i(n) \in \N$ be the largest integer such that $d^i = o(n)$. 
Suppose that $d^{i+1} / n - 2 \log n \to \infty$.

Then, the following property holds a.a.s.\ for $G = (V,E) \in \Gnp$:
$$
\pqn{G} \le k := (2+o(1)) \frac {n \log n}{d^i} = \bigo \left( \frac {n \log n}{d^i} \right).
$$
\end{theorem}
\begin{proof}
Note that we assumed that $d^{i+1} / n - 2 \log n \to \infty$ and by definition of $i$ we get that $d^{i} / n - 2 \log n \to -\infty$. Hence, by Lemma~\ref{lem:diameter} we get that the diameter of $\Gnp$ is equal to $i+1$ a.a.s.

\medskip

Let $\om = \om(n) = \min( \sqrt{\omega}, n/d^i )$. Clearly, $\om \to \infty$ as $n \to \infty$. We say that a graph $G=(V,E)$ is \emph{nice} if for any $x,y \in V$ ($x \neq y$) we have
\begin{eqnarray*}
\left| N_i(x) \setminus N_{\le i}(y) \right| &=& \left(1+ \bigo \left( \om^{-1} \right) \right) d^i \sim d^i, \\
\left| N_{\le i}(x) \cap N_{\le i}(y) \right| &=& \bigo \left( \om^{-1} \right)  d^i = o(d^i).
\end{eqnarray*}
We will show that for any nice graph $G$, $\pqn{G} \le k$ (deterministically). This will finish the proof since, by Lemma~\ref{lem:gnp exp2}(a), a.a.s.\ $G \in \Gnp$ is nice.

Let us then move away from random graphs and concentrate on any (deterministic) nice graph $G=(V,E)$ on $n=|V|$ vertices. The first player (the algorithm) will use a randomized strategy to play the game for the first $k$ rounds. We will show that this strategy is a winning one with probability at least $1/2$. By the trivial probabilistic method, it will imply that there exists a winning strategy and the bound $\pqn{G} \le k$ will be established. To that end, it is enough to show that for any pair of vertices $x$ and $y$, regardless of what the second player (the adversary) does, the randomized strategy eliminates at least one of them from being the target with probability at least $1-n^{-2}$. Formally, we need to show that
\begin{equation}\label{eq:distinguish_pair}
\Prob \left( \{x,y\} \subseteq V_k \right) \le n^{-2} \qquad \text{for any $x,y \in V$}.
\end{equation}
Indeed, if~\eqref{eq:distinguish_pair} is established, then by the union bound
$$
\Prob(\text{the strategy fails}) = \Prob \left( |V_k| \ge 2 \right) = \Prob \left( \exists x,y \in V : \{x,y\} \subseteq V_k \right) \le \binom{n}{2} n^{-2} \le 1/2.
$$

\medskip

Let us now describe the randomized strategy the algorithm is going to use. Fix $b > 1$, an arbitrarily large but fixed real number. The strategy consists of $a \log n$ independent \emph{phases}, where $a = 2/(b-1)$. Each phase will last for $b n / d^i$ rounds so the total number of rounds is equal to 
$$
k = (a \log n) \, \frac {b n}{d^i} = \frac {2b}{b-1} \, \frac {n \log n}{d^i}.
$$
Since $2b/(b-1) \to 2$ as $b\to \infty$, we will get the desired upper bound. 

At the beginning of each phase, the algorithm generates a sequence of $bn/d^i + 1$ random vertices, each of them selected independently and uniformly at random from $V$, all available vertices. During the first round, the algorithm presents the first two vertices from the sequence. Then, in each consecutive round, the algorithm presents the vertex the adversary replied with in the previous round together with the next vertex from the random sequence. 

\medskip

It remains to show that~\eqref{eq:distinguish_pair} holds. Let us fix any $x,y \in V$. We partition the set of vertices $V$ into the following 4 parts:
\begin{eqnarray*}
X &=& N_i(x) \setminus N_{\le i}(y) \\
Y &=& N_i(y) \setminus N_{\le i}(x) \\
Z &=& V \setminus N_{\le i}(\{x,y\}) \\
R &=& V \setminus (X \cup Y \cup Z) = (N_i(x) \cap N_i(y)) \cup N_{\le i-1}(\{x,y\}).  
\end{eqnarray*}
Since the diameter of $G$ is $i+1$, we get that
\begin{itemize}
\item [(i)] vertices in $X$ are at distance $i$ from $x$ and at distance $i+1$ from $y$, 
\item [(ii)] vertices in $Y$ are at distance $i+1$ from $x$ and at distance $i$ from $y$, 
\item [(iii)] vertices in $Z$ are at distance $i+1$ from both $x$ and $y$, 
\item [(iv)] vertices in $R$ are at distance $i$ from both $x$ and $y$ or less than $i$ for at least one of $x$ of $y$.
\end{itemize}
Since $G$ is nice, we have that both $X$ and $Y$ consist of $\left(1+ \bigo \left( \om^{-1} \right) \right) d^i \sim d^i$ vertices and $|R| = \bigo \left( \om^{-1} \right) d^i = o(d^i)$. As a result, $|Z| = n - (2+o(1)) d^i \sim n$.

A given phase is \emph{good} if the following three properties hold for the random sequence of $bn/d^i + 1$ vertices the algorithm generates at the beginning of that phase:
\begin{itemize}
\item [(1)] the sequence consists of only vertices from $X \cup Y \cup Z$,
\item [(2)] the sequence has at least one vertex from $X$ and at least one vertex from $Y$,
\item [(3)] the first vertex in the sequence is from $Z$.
\end{itemize}
First, let us note that if a phase is good, then the strategy guarantees that at least one of $x$ and $y$ is eliminated during that phase. Indeed, the properties (1)--(3) of being good, and the strategy used, imply that at some round during this phase, $(u_t, v_t)$ is presented with $u_t \in Z$ and $v_t \in X \cup Y$. By symmetry, without loss of generality, we may assume that $v_t \in X$. If the adversary replies with $u_t$, then $x$ gets eliminated: $x$ is closer to $v_t$ than to $u_t$. So assume that the adversary replies with $v_t$ which does not eliminate any of the two vertices: $x$ is closer to $v_t$ than to $u_t$ and $y$ is at the same distance from both of them. Vertex $v_t$ is kept by the algorithm and reused in the next round. During the following rounds, $v_t$ can be presented together with some other vertices from $Z$ but the only chance for $x$ not to be eliminated is that the adversary keeps replying with a vertex from $X$. (Trivially, a vertex from $X$ is transferred to the next round when two vertices from $X$ are presented.) However, eventually, a vertex from $X$ will be presented together with a vertex from $Y$ eliminating either $x$ or $y$.

Now, we will show that at least one phase must be good with probability at least $1-n^{-2}$. Let $\zeta$ be the random variable counting the number of vertices from $R$ that occur during a given phase. Clearly, 
$$
\E [\zeta] = \frac {|R|}{n} \left( \frac {bn}{d^i} + 1 \right) = \bigo \left( \frac {d^i}{n \om} \, \frac {bn}{d^i} \right) = \bigo ( \om^{-1} ) = o(1). 
$$
Hence, by the first moment method, the probability that a given phase does not satisfy property~(1) can be estimated as follows:
$$
\Prob ( \zeta \ge 1 ) \le \E [\zeta] = \bigo ( \om^{-1} ) = o(1).
$$
On the other hand, by the union bound, the probability that a given phase does not satisfy property~(2) is at most 
$$
\left( 1 - \frac {|X|}{n} \right)^{bn/d^i + 1} + \left( 1 - \frac {|Y|}{n} \right)^{bn/d^i + 1} \le 2 \exp \left( - \frac {d^i}{n} \, \frac {bn}{d^i} (1+o(1)) \right) = 2e^{-b} + o(1).
$$
Trivially, property~(3) is not satisfied with probability $o(1)$. Hence, all independent phases fail to be good with probability at most
\begin{eqnarray*}
\left( 2e^{-b} + o(1) \right)^{a \log n} &=& \exp \left( \Big( ( \log 2 - b ) a + o(1) \Big) \log n \right) \\
&=& \exp \left( -2 \Big( \frac {b-\log 2}{b-1} + o(1) \Big) \log n \right) \\
&\le& \exp \left( -2 \log n \right) = 1/n^{-2}.
\end{eqnarray*}
This finishes the proof of~\eqref{eq:distinguish_pair} and so also the entire proof of the theorem. 
\end{proof}

Now, let us adjust the proof of Theorem~\ref{thm:upper_bound_pqn} to claim the same upper bound for $\eqn{\Gnp}$.

\begin{theorem}\label{thm:upper_bound_eqn}
Suppose that $\omega \log n \le d=pn \le n / \omega = o(n)$, where $\omega=\omega(n)$ is any function tending to infinity as $n \to \infty$ such that $\omega \le \log n$.  
Let $i=i(n) \in \N$ be the largest integer such that $d^i = o(n)$. let $\om = \om(n) = \min( \sqrt{\omega}, n/d^i ) \to \infty$ as $n \to \infty$.
Finally, suppose that $d^{i+1} / n \gg \om \log n$.

Then, the following property holds a.a.s.\ for $G = (V,E) \in \Gnp$:
$$
\eqn{G} \le k := (2+o(1)) \frac {n \log n}{d^i} = \bigo \left( \frac {n \log n}{d^i} \right).
$$
\end{theorem}
\begin{proof}
We will use the same notation as in the proof of Theorem~\ref{thm:upper_bound_pqn}. As before, we will have $a \log n$ independent phases, each phase lasting for $bn/d^i$ rounds. Since the first player (the algorithm) has to present edges, we need to modify the strategy slightly. During the first round of each phase, the algorithm presents a vertex selected uniformly at random from $V$. Then, in each consecutive round, the algorithm presents the vertex the adversary replied with in the previous round together with a random neighbour of that vertex. That is the only modification that is required. As before, we investigate any deterministic nice graph and independently consider each pair of vertices $x,y$ ($x \neq y$). The goal, as before, is to show that with probability at least $1-n^{-2}$, at the end of the last phase the target cannot be hidden in both $x$ and $y$. 

As before, a given phase is good if only vertices from $X \cup Y \cup Z$ are played during that phase, at least one vertex from $X$ and at least one vertex from $Y$ were played, and the first vertex is from $Z$. 
By Lemma~\ref{lem:gnp exp}(a), we may assume that each vertex has degree asymptotic to $d$. To show that a.a.s.\ at least one phase must be good, and so a.a.s.\ the strategy is a winning strategy, we will use Lemma~\ref{lem:gnp exp2}. Let us concentrate on any given phase. By property~(c), a.a.s.\ the first vertex is selected from $Z$. 
By property~(e), provided that the first endpoint of an edge is in $X \cup Y \cup Z$, the second endpoint is in $R$ with probability $o( d^{i+1} n^{-1} ) / d = o( d^{i} / n )$, which is of the same order as $|R|/n$. Since there are $\bigo (n/d^i)$ rounds in that phase, as in the previous proof, a.a.s.\ no vertex from $R$ is presented during that phase.
Similarly, by properties~(f)--(g), provided that the first endpoint of an edge is in $X \cup Z$ ($X \cup Z$, respectively), the second endpoint is in $Y$ ($X$, respectively) with probability asymptotic to $(d^{i+1} n^{-1})/d \sim d^i/n$, and so asymptotically the same as $|Y|/n$ ($|X|/n$, respectively).
Hence, conditioning on the fact that no vertex from $R$ is presented, the probability that no vertex from $Y$ ($X$, respectively) is presented is at most
$$
\left( 1 - (1+o(1)) \frac {d^i}{n} \right)^{bn/d^i + 1} \le e^{-b} + o(1).
$$
We get that a given phase fails with probability at most $2e^{-b} + o(1)$ and the argument continues as before. This finishes the proof of the theorem.
\end{proof}

Finally, let us deal with very dense random graphs, that is, when $p \in (0,1)$ is a fixed constant. The following observation follows immediately from the Chernoff's bound~\eqref{chern} (applied with $\eps = c \sqrt{\log n / n}$, where $c=c(p)$ is a sufficiently large constant) and the union bound. Since the reader is already warmed-up, we skip a simple proof and leave it as an exercise. 

\begin{lemma}\label{lem:gnp exp3}
Fix $p \in (0,1)$. 
The following properties hold a.a.s.\ for $G = (V,E) \in \Gnp$. 
\begin{itemize}
\item [(a)] For any vertex $x \in V$, $\deg(x) = |N_1(x)| = (1+\bigo( \sqrt{\log n / n} )) \, pn \sim pn$.
\item [(b)] For any vertices $x,y \in V$ ($x \neq y$), $|N_1(x) \setminus N_{\le 1}(y)| = (1+\bigo( \sqrt{\log n / n} )) \, p(1-p)n \sim p(1-p)n$.
\item [(c)] For any vertices $x,y,u \in V$ (pairwise distinct), 
$$
| N_1(u) \cap (N_1(y) \setminus N_{\le 1}(x)) | = (1+\bigo( \sqrt{\log n / n} )) \, p^2(1-p)n \sim p^2(1-p)n.
$$
\end{itemize}
\end{lemma}

Now, we are ready to prove our last upper bound. 

\begin{theorem}\label{thm:dense_graphs}
Fix $p \in (0,1)$. 
The following property holds a.a.s.\ for $G = (V,E) \in \Gnp$:
$$
\pqn{G} \le \eqn{G} \le k := \frac {3}{p^2 (1-p)^2} \log n = \bigo \left( \log n \right).
$$
\end{theorem}
\begin{proof}
We continue using the same proof strategy as in the previous two theorems, Theorems~\ref{thm:upper_bound_pqn} and~\ref{thm:upper_bound_eqn}, but the strategy and the proof that it is a winning strategy a.a.s.\ are much easier than before. There are $a \log n$ rounds, where $a = 3 p^{-2} (1-p)^{-2}$. In each round, the first player (the algorithm) presents an edge that is generated at random, that is, the first vertex, $u$, is selected uniformly at random from $V$ and then a neighbour of $u$, vertex $v$, is selected uniformly at random from all neighbours of $u$. 
As in the previous two proofs, since we aim for a statement that holds a.a.s., we may assume that properties stated in Lemma~\ref{lem:gnp exp3} hold deterministically (that is, our graph is ``nice''). We need to show that for any pair of vertices $x,y$ ($x \neq y$), the above strategy guarantees that with probability at least $1-n^{-2}$ the target cannot be hidden in both vertices. 

Let us fix any pair of vertices $x,y$ ($x \neq y$). If $u \in N_1(x) \setminus N_{\le 1}(y)$ but $v \in N_1(y) \setminus N_{\le 1}(x)$ (or vice-versa), then the goal is achieved, regardless what the second player (the adversary) does. By Lemma~\ref{lem:gnp exp3}(b), $u \in N_1(x) \setminus N_{\le 1}(y)$ with probability asymptotic to $p(1-p)$. Conditioning on this event, by Lemma~\ref{lem:gnp exp3}(c), $u$ has $(1+o(1)) p^2 (1-p) n$ neighbours in $N_1(y) \setminus N_{\le 1}(x)$. Hence, by Lemma~\ref{lem:gnp exp3}(a), the (conditional) probability that $v \in N_1(y) \setminus N_{\le 1}(x)$ is asymptotic to $p (1-p)$. We get that $u \in N_1(x) \setminus N_{\le 1}(y)$ but $v \in N_1(y) \setminus N_{\le 1}(x)$ happens with probability asymptotic to $p^2(1-p)^2$. By symmetry, the same is true if $u$ and $v$ are swapped. It follows that the strategy fails with probability at most
$$
\left( 1 - (2+o(1)) p^2 (1-p)^2 \right)^{a \log n} \le \exp \left( - (2+o(1)) p^2 (1-p)^2 a \log n \right) \le n^{-2}, 
$$
since $a = 3 p^{-2} (1-p)^{-2}$. This finishes the proof of the theorem.
\end{proof}

\section{$\NP$-hardness}\label{sec:hardness}

In this section, we prove that computing the parameters $\pqn{G}$ and $\eqn{G}$ is $\NP$-hard for a general family of graphs, that is, for graphs of diameter at most $3$. This is done by a reduction from the $\SetCover$ problem: given a set of elements $A = \{a_1, \dots, a_n\}$, where $n = 3k$ for some natural number $k$, a family of sets $\mathcal{S} = \{S_1, \dots, S_m\},$ for $S_i \subseteq A$, $|S_i| = 3$ for $i \in [m]$ and $\bigcup_{i \in [m]} S_i = A$, the question is whether there exists a sub-family $\mathcal{SC} = \{S_{i_1}, \dots, S_{i_k}\}$ such that $\bigcup_{j \in [k]} S_{i_j} = A$. Note that, since $n = 3k$, if the answer to the above question is positive, then the sets in $\mathcal{SC}$ are pairwise disjoint.

For each instance of the problem, we construct a graph $G$ defined as follows. The vertex set of $G$ is defined as follows:
$$
V(G) = \{x, y\} \cup \left( \bigcup_{t \in [5]} A^t \right) \cup \left( \bigcup_{t \in [5]} \mathcal{S}^t \right) \cup L, 
$$
where $A^t, \mathcal{S}^t$, for $t \in [5]$, are copies of the sets $A$ and $\mathcal{S}$, respectively, and $|L| = 5n (5(m-k)-3) = n (25 (m-k) - 15)$.  The edge set of $G$ is defined as follows (see Figure~\ref{fig:reduct} for an illustration):
\begin{enumerate}
    \item $G\left[\{y\} \cup \bigcup_{t \in [5]} A^t \right]$ induces a star centred at $y$ (isomorphic to $K_{1, 5n}$), 
    \item each vertex $a \in \bigcup_{t \in [5]} A^t$ is adjacent to $5(m-k)-3$ leaves that belong to $L$,
    \item $G\left[\{x\} \cup \bigcup_{t \in [5]} \mathcal{S}^t \right]$ induces a star centered at $x$ (isomorphic to $K_{1, 5m}$),
    \item for $t \in [5], i \in [n], j \in [m]$, vertices $a^t_i \in A^t$ and $S^t_j \in \mathcal{S}^t$ are adjacent if and only if $a_i \in S_j$ in the corresponding instance of the problem. 
\end{enumerate}
For any $i \in [n]$ and $t \in [5]$, we define $L^t_i = \{a^t_i\} \cup \left(N_1 \left(a^t_i\right) \cap L\right)$, that is, the vertex set of a star centred at $a^t_i$ (isomorphic to $K_{1, 5(m-k)-3}$) defined by the second condition above. 

\begin{figure}[htb]
	\begin{center}
		\includegraphics[width = 0.8\textwidth]{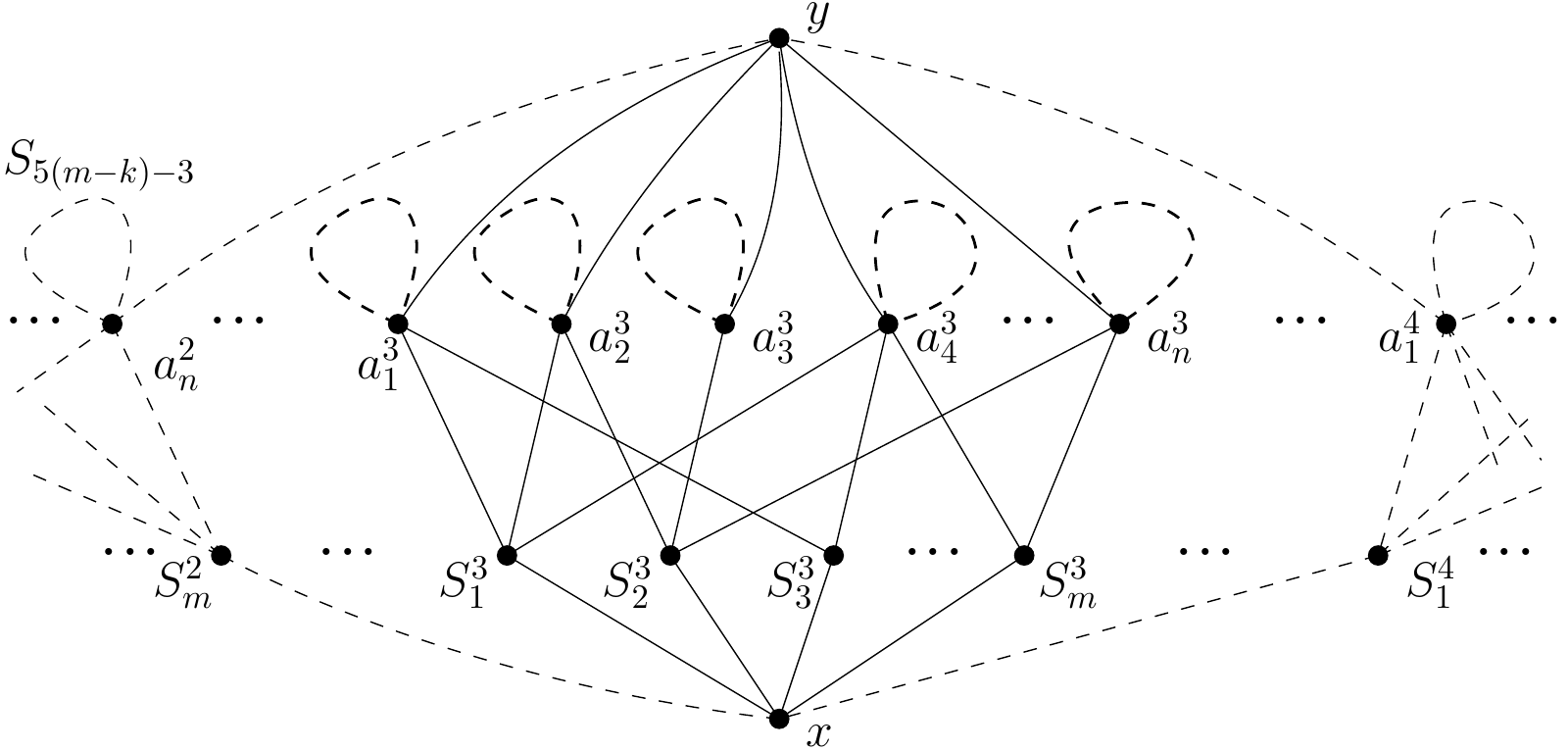}
	\end{center}
	\caption{A reduction from the $3$-EXACT SET COVER}
	\label{fig:reduct}
\end{figure}

\begin{lemma}\label{lem:NP}
Suppose that $n > m > k + 3 = n/3+3$ for some natural number $k$. For the graph $G$ defined above, one has $\pqn{G} \le \eqn{G} \le 5m$ 
if and only if there exists a set cover of size $k$ for the corresponding instance of the $\SetCover$ problem.
\end{lemma}
\begin{proof}
($\Longleftarrow$) Suppose that the desired set cover $\mathcal{SC}= \{S_{i_1}, \dots, S_{i_k}\}$ exists. We introduce the following search strategy for the first player (the algorithm) implying that $\eqn{G} \le 5m$. Query the $5k$ edges of the form $x S^t_{i_j}$ for $t \in [5], j \in [k]$. If every single reply is $x$, then the target is on a star induced by the graph $G\left[\{x\} \cup \bigcup_{t \in [5]} \mathcal{S}^t  \setminus \bigcup_{t \in [5], j \in [k]} \left\{S^t_{i_j}\right\}\right]$ (isomorphic to $K_{1, 5(m-k)}$), and can be easily localized in the next $5(m-k)$ rounds. (Actually, it can be done faster by querying edges of the form $y a^t_i$ for some carefully selected values of $t \in [5]$ and $i \in [n]$.) The target is found in at most $5m$ rounds in total, yielding the desired upper bound for $\eqn{G}$. 
    
Suppose then that the adversary replies with vertex $S^t_{i_j}$ at some round of the initial $5k$ queries. Without loss of generality, we may assume that it was $S^1_1$ and that this vertex is adjacent to $a^1_1, a^1_2$, and $a^1_3$. It is easy to see that the target must be in $S^1_1$, $y$, or at a vertex of some star $L^1_i$ centered at $a^1_i$ ($i \in [3]$). To locate it, the algorithm queries the 3 edges adjacent to $S^1_1$, namely, edges $S^1_1 a^1_i$ ($i \in [3]$). If $S^1_1$ is the target, then $S^1_1$ is replied 3 times in a row. On the other hand, if $y$ is the target, then $S^1_1$ is not selected by the adversary even once. Suppose then that the target is located in some vertex of a star $L^1_i$ ($i \in [3]$). Clearly, the adversary replies with $a^1_i$ when this vertex is queried and, otherwise, $S^1_1$ is selected. The algorithm learns which star needs to be investigated and the target can be found in the remaining $5(m-k)-3$ queries.
This proves that $\eqn{G}\leq 5m$ and the claim follows immediately from Observation~\ref{obs:lower}.

\medskip

($\Longrightarrow$) We will prove the implication using proof by contraposition. Suppose that there is no set cover of size $k$ for the corresponding instance of the problem. Our goal is to show that $\pqn{G} \ge 5m+1$. To that end, we will provide a strategy for the second player (the adversary) that guarantees that the target is not found after $5m$ rounds of the game, that is, $|V_{5m}| \ge 2$. In fact, we will help the first player (the algorithm) slightly and announce at the beginning of the game that the target is hidden in the set 
$$
V_0 = L \cup \bigcup_{t \in [5]} A^t = \bigcup_{i \in [n], t \in [5]} L_i^t \subseteq V
$$ 
(instead of starting with $V_0 = V$). Note that $G[V_0]$ is a family of $5n$ independent stars, each isomorphic to $K_{1, 5(m-k)-3}$. Even with that additional information, the algorithm will not be able to find the target in $5m$ rounds.

We consider two phases of the game. The first phase lasts for $5k + 4$ rounds. We will show that at the end of this phase, there exists at least one star $L_i^t$ with the property that the target can be at any vertex of that star; that is, $L_i^t \subseteq V_{5k+4}$. Here is a strategy for the adversary. Suppose that for some $i \in [5k + 4]$ the algorithm queried a pair $u_i$ and $v_i$ at round $i$.
\begin{enumerate}
    \item $u_i = y$ and $v_i = a^t_j$ for some $j \in [n], t \in [5]$. The adversary replies with $y$. Then, $V_{i-1} \setminus V_i = V_{i-1} \setminus \Cl(y, a^t_j) \subseteq L^t_j$. At most one star is eliminated.
    \item $u_i = y$ and $v_i \not \in \bigcup_{t \in [5]} A^t$. The adversary replies with $y$. Then, only $v_i$ is eliminated, provided $v_i \in L$. At most one star is (partially) eliminated.
    \item $u_i = x$ and $v_i \in L^t_j$ for some $j \in [n], t \in [5]$. The adversary replies with $x$. Then, $V_{i-1} \setminus V_i = V_{i-1} \setminus \Cl(x, v_i) \subseteq L^t_j$. At most one star is eliminated.
    \item $u_i = x$ and $v_i = S^t_\ell$ for some $\ell \in [m], t \in [5]$. Let $S_\ell = \{a_{\ell_1}, a_{\ell_2}, a_{\ell_3}\}$. The adversary replies with $x$. Then, $V_{i-1} \setminus V_i = V_{i-1} \setminus \Cl(x, S^t_\ell) \subseteq L^t_{\ell_1} \cup L^t_{\ell_2} \cup L^t_{\ell_3}$. At most three stars are eliminated.
    \item $u_i = S^t_\ell$ and $v_i = S^s_j$ for some $\ell, j \in [m]$ and $t, s \in [5]$. Let $S_\ell = \{a_{\ell_1}, a_{\ell_2}, a_{\ell_3}\}$. The adversary replies with $v_i$. Then, $V_{i-1} \setminus V_i = V_{i-1} \setminus \Cl(v_i, S^t_\ell) \subseteq L^t_{\ell_1} \cup L^t_{\ell_2} \cup L^t_{\ell_3}$. At most 3 stars are eliminated. Note that this is equivalent to the previous case ($x$ vs.\ $S^t_\ell$).
    \item $u_i = S^t_\ell$ and $v_i \in L^s_j$ for some $\ell, j \in [m]$ and $t, s \in [5]$. Let $S_\ell = \{a_{\ell_1}, a_{\ell_2}, a_{\ell_3}\}$. The adversary replies with $v_i$. Then, $V_{i-1} \setminus V_i = V_{i-1} \setminus \Cl(v_i, S^t_\ell) \subseteq L^t_{\ell_1} \cup L^t_{\ell_2} \cup L^t_{\ell_3}$. At most 3 stars are eliminated. Again, note that this is equivalent to the previous case ($x$ vs.\ $S^t_\ell$), provided $t \neq s$ or $a_j \not \in S_\ell$. If $t = s$ and $a_j \in S_\ell$, then at most two stars are eliminated.
    \item $u_i = a^t_j$ and $v_i \in L^s_\ell$ for some $j, \ell \in [m]$ and $t, s \in [5]$. The adversary replies with $a^t_j$. Then $V_{i-1} \setminus V_i = V_{i-1} \setminus \Cl(a^t_j, v_i) \subseteq L^s_\ell$, that is, at most one star is eliminated. When $s = t$ and $j = \ell$ then only $v_i$ is eliminated. In any case, at most one star is (partially) eliminated.
    \item $\{u_i, v_i\} \in L$. Let $v_i \in L^t_j$ for some $j \in [n], t \in [5]$. The adversary replies with $u_i$. Then $V_{i-1} \setminus V_i = V_{i-1} \setminus \Cl(u_i, v_i) \subseteq L^t_j$. (In particular, if $u_i \in L^t_j$, then only $v_i$ is eliminated.) At most one star is (partially) eliminated. 
\end{enumerate}

Note that each query described above eliminates at most three stars, either entirely or partially. Moreover, all stars that are eliminated belong to one of the five copies of the sets $A$. More importantly, since there is no cover set of size $k$ for the corresponding instance of the problem, the algorithm cannot eliminate all $3k$ stars from one copy in $k$ queries (restricted to those that affect this copy). Hence, after $5k+4$ rounds, there is at least one star (say, $L^t_j$ centered at $a^t_j$) that may hide the target, that is, $L^t_j \subseteq V_{5k+4}$. 

The strategy for the second phase is very easy. In every round of this phase, the adversary replies with a vertex that is closer to $a^t_j$ (or any vertex if both vertices that are queried are at the same distance from $a^t_j$). Clearly, at most one vertex of $L^t_j$ is eliminated in each round. Hence, since the star consists of $5(m-k)-2$ vertices, after $5(m-k)-4$ additional rounds, at least two vertices are still left in $V_{5m}$ and so the game is not over yet. This finishes the proof that $\pqn{G} \ge 5m+1$, and we are done.
\end{proof}

Since both our search problems are in $\NP$ (the strategy length is trivially bounded from above by the order of the graph; see also Observation~\ref{obs:lower}), Lemma~\ref{lem:NP} implies immediately Theorem~\ref{thm:NP}.

\bibliography{ref}

\begin{thebibliography}{10}

\bibitem{BEHAGUE202280}
Natalie~C. Behague, Anthony Bonato, Melissa~A. Huggan, Trent~G. Marbach, and
  Brittany Pittman.
\newblock The localization capture time of a graph.
\newblock {\em Theoretical Computer Science}, 911:80--91, 2022.
\newblock \href {https://doi.org/https://doi.org/10.1016/j.tcs.2022.02.007}
  {\path{doi:https://doi.org/10.1016/j.tcs.2022.02.007}}.

\bibitem{Ben-AsherFN99}
Yosi Ben{-}Asher, Eitan Farchi, and Ilan Newman.
\newblock Optimal search in trees.
\newblock {\em {SIAM} J. Comput.}, 28(6):2090--2102, 1999.
\newblock \href {https://doi.org/10.1137/S009753979731858X}
  {\path{doi:10.1137/S009753979731858X}}.

\bibitem{Ben-OrH08}
Michael Ben{-}Or and Avinatan Hassidim.
\newblock The bayesian learner is optimal for noisy binary search (and pretty
  good for quantum as well).
\newblock In {\em {FOCS}}, pages 221--230, 2008.
\newblock \href {https://doi.org/10.1109/FOCS.2008.58}
  {\path{doi:10.1109/FOCS.2008.58}}.

\bibitem{bollobas1981diameter}
B{\'e}la Bollob{\'a}s.
\newblock The diameter of random graphs.
\newblock {\em Transactions of the American Mathematical Society},
  267(1):41--52, 1981.

\bibitem{Bollobas}
B{\'e}la Bollob{\'a}s.
\newblock {\em Random graphs}.
\newblock Cambridge University Press, 2001.

\bibitem{bollobas2013metric}
B{\'e}la Bollob{\'a}s, Dieter Mitsche, and Pawe{\l} Pra{\l}at.
\newblock Metric dimension for random graphs.
\newblock {\em The Electronic Journal of Combinatorics}, 20(4):P1, 2013.

\bibitem{bonato2020bounds}
Anthony Bonato and William~B Kinnersley.
\newblock Bounds on the localization number.
\newblock {\em Journal of Graph Theory}, 94(4):579--596, 2020.

\bibitem{bonato2017graph}
Anthony Bonato and Pra\l{}at Pawe\l{}.
\newblock {\em Graph searching games and probabilistic methods}.
\newblock Chapman and Hall/CRC, 2017.

\bibitem{bosek2018centroidal}
Bart{\l}omiej Bosek, Przemys{\l}aw Gordinowicz, Jaros{\l}aw Grytczuk, Nicolas
  Nisse, Joanna Sok{\'o}{\l}, and Ma{\l}gorzata {\'S}leszy{\'n}ska-Nowak.
\newblock Centroidal localization game.
\newblock {\em The Electronic Journal of Combinatorics}, pages P4--62, 2018.

\bibitem{bosek2018localization}
Bart{\l}omiej Bosek, Przemys{\l}aw Gordinowicz, Jaros{\l}aw Grytczuk, Nicolas
  Nisse, Joanna Sok{\'o}{\l}, and Ma{\l}gorzata {\'S}leszy{\'n}ska-Nowak.
\newblock Localization game on geometric and planar graphs.
\newblock {\em Discrete Applied Mathematics}, 251:30--39, 2018.

\bibitem{CarmoDKL04}
Renato Carmo, Jair Donadelli, Yoshiharu Kohayakawa, and Eduardo~Sany Laber.
\newblock Searching in random partially ordered sets.
\newblock {\em Theor. Comput. Sci.}, 321(1):41--57, 2004.
\newblock \href {https://doi.org/10.1016/j.tcs.2003.06.001}
  {\path{doi:10.1016/j.tcs.2003.06.001}}.

\bibitem{carraher2012locating}
James Carraher, Ilkyoo Choi, Michelle Delcourt, Lawrence~H Erickson, and
  Douglas~B West.
\newblock Locating a robber on a graph via distance queries.
\newblock {\em Theoretical Computer Science}, 463:54--61, 2012.

\bibitem{CicaleseJLV12}
Ferdinando Cicalese, Tobias Jacobs, Eduardo~Sany Laber, and Caio~Dias Valentim.
\newblock The binary identification problem for weighted trees.
\newblock {\em Theor. Comput. Sci.}, 459:100--112, 2012.
\newblock \href {https://doi.org/10.1016/j.tcs.2012.06.023}
  {\path{doi:10.1016/j.tcs.2012.06.023}}.

\bibitem{CicaleseKLPV16}
Ferdinando Cicalese, Bal{\'{a}}zs Keszegh, Bernard Lidick{\'{y}},
  D{\"{o}}m{\"{o}}t{\"{o}}r P{\'{a}}lv{\"{o}}lgyi, and Tom{\'{a}}\v{s} Valla.
\newblock On the tree search problem with non-uniform costs.
\newblock {\em Theor. Comput. Sci.}, 647:22--32, 2016.
\newblock \href {https://doi.org/10.1016/j.tcs.2016.07.019}
  {\path{doi:10.1016/j.tcs.2016.07.019}}.

\bibitem{DaganFGM17}
Yuval Dagan, Yuval Filmus, Ariel Gabizon, and Shay Moran.
\newblock Twenty (simple) questions.
\newblock In {\em {STOC}}, pages 9--21, 2017.
\newblock \href {https://doi.org/10.1145/3055399.3055422}
  {\path{doi:10.1145/3055399.3055422}}.

\bibitem{Dereniowski06}
Dariusz Dereniowski.
\newblock Edge ranking of weighted trees.
\newblock {\em Discrete Applied Mathematics}, 154(8):1198--1209, 2006.
\newblock \href {https://doi.org/10.1016/j.dam.2005.11.005}
  {\path{doi:10.1016/j.dam.2005.11.005}}.

\bibitem{DereniowskiKUZ17}
Dariusz Dereniowski, Adrian Kosowski, Przemys\l{}aw Uzna\'{n}ski, and Mengchuan
  Zou.
\newblock Approximation strategies for generalized binary search in weighted
  trees.
\newblock In {\em {ICALP}}, pages 84:1--84:14, 2017.
\newblock \href {https://doi.org/10.4230/LIPIcs.ICALP.2017.84}
  {\path{doi:10.4230/LIPIcs.ICALP.2017.84}}.

\bibitem{DereniowskiLU21}
Dariusz Dereniowski, Aleksander Lukasiewicz, and Przemyslaw Uznanski.
\newblock Noisy searching: simple, fast and correct.
\newblock {\em CoRR}, abs/2107.05753, 2021.
\newblock URL: \url{https://arxiv.org/abs/2107.05753}.

\bibitem{DereniowskiTUW19}
Dariusz Dereniowski, Stefan Tiegel, Przemyslaw Uznanski, and Daniel
  Wolleb{-}Graf.
\newblock A framework for searching in graphs in the presence of errors.
\newblock In {\em {SOSA@SODA}}, pages 4:1--4:17, 2019.
\newblock \href {https://doi.org/10.4230/OASIcs.SOSA.2019.4}
  {\path{doi:10.4230/OASIcs.SOSA.2019.4}}.

\bibitem{dudek2022localization}
Andrzej Dudek, Sean English, Alan Frieze, Calum MacRury, and Pawe{\l}
  Pra{\l}at.
\newblock Localization game for random graphs.
\newblock {\em Discrete Applied Mathematics}, 309:202--214, 2022.

\bibitem{dudek2019note}
Andrzej Dudek, Alan Frieze, and Wesley Pegden.
\newblock A note on the localization number of random graphs: diameter two
  case.
\newblock {\em Discrete Applied Mathematics}, 254:107--112, 2019.

\bibitem{Zadeh17}
Ehsan Emamjomeh{-}Zadeh and David Kempe.
\newblock A general framework for robust interactive learning.
\newblock In {\em {NIPS}}, pages 7085--7094, 2017.

\bibitem{Zadeh16}
Ehsan Emamjomeh{-}Zadeh, David Kempe, and Vikrant Singhal.
\newblock Deterministic and probabilistic binary search in graphs.
\newblock In {\em {STOC}}, pages 519--532, 2016.
\newblock \href {https://doi.org/10.1145/2897518.2897656}
  {\path{doi:10.1145/2897518.2897656}}.

\bibitem{FeigeRPU94}
Uriel Feige, Prabhakar Raghavan, David Peleg, and Eli Upfal.
\newblock Computing with noisy information.
\newblock {\em {SIAM} J. Comput.}, 23(5):1001--1018, 1994.
\newblock \href {https://doi.org/10.1137/S0097539791195877}
  {\path{doi:10.1137/S0097539791195877}}.

\bibitem{foucaud2014centroidal}
Florent Foucaud, Ralf Klasing, and Peter~J Slater.
\newblock Centroidal bases in graphs.
\newblock {\em Networks}, 64(2):96--108, 2014.

\bibitem{GiannopoulouHT12}
Archontia~C. Giannopoulou, Paul Hunter, and Dimitrios~M. Thilikos.
\newblock Lifo-search: {A} min-max theorem and a searching game for cycle-rank
  and tree-depth.
\newblock {\em Discret. Appl. Math.}, 160(15):2089--2097, 2012.
\newblock \href {https://doi.org/10.1016/j.dam.2012.03.015}
  {\path{doi:10.1016/j.dam.2012.03.015}}.

\bibitem{melter1976metric}
F.~Harary and R.A. Melter.
\newblock On the metric dimension of a graph.
\newblock {\em Ars Combin.}, 2:191--195, 1976.

\bibitem{JLR}
Svante Janson, Tomasz \L{}uczak, and Andrzej Ruci\'nski.
\newblock {\em Random graphs}, volume~45.
\newblock John Wiley \& Sons, 2011.

\bibitem{Karonski_Frieze}
Michal Karo\'nski and Alan Frieze.
\newblock {\em Introduction to Random Graphs}.
\newblock Cambridge University Press, 2016.

\bibitem{KarpK07}
Richard~M. Karp and Robert Kleinberg.
\newblock Noisy binary search and its applications.
\newblock In {\em {SODA}}, pages 881--890, 2007.

\bibitem{LaberMP02}
Eduardo~Sany Laber, Ruy~Luiz Milidi{\'{u}}, and Artur~Alves Pessoa.
\newblock On binary searching with nonuniform costs.
\newblock {\em {SIAM} J. Comput.}, 31(4):1022--1047, 2002.
\newblock \href {https://doi.org/10.1137/S0097539700381991}
  {\path{doi:10.1137/S0097539700381991}}.

\bibitem{LamY01}
Tak~Wah Lam and Fung~Ling Yue.
\newblock Optimal edge ranking of trees in linear time.
\newblock {\em Algorithmica}, 30(1):12--33, 2001.
\newblock \href {https://doi.org/10.1007/s004530010076}
  {\path{doi:10.1007/s004530010076}}.

\bibitem{lichev2021localization}
Lyuben Lichev, Dieter Mitsche, and Pawe{\l} Pra{\l}at.
\newblock Localization game for random geometric graphs.
\newblock {\em European Journal of Combinatorics}, 108:103616, 2023.

\bibitem{Liu90}
Joseph~W.H. Liu.
\newblock The role of elimination trees in sparse factorization.
\newblock {\em {SIAM} Journal on Matrix Analysis and Applications},
  11(1):134--172, 1990.
\newblock \href {https://doi.org/10.1137/0611010} {\path{doi:10.1137/0611010}}.

\bibitem{MakinoUI01}
Kazuhisa Makino, Yushi Uno, and Toshihide Ibaraki.
\newblock On minimum edge ranking spanning trees.
\newblock {\em J. Algorithms}, 38(2):411--437, 2001.
\newblock \href {https://doi.org/10.1006/jagm.2000.1143}
  {\path{doi:10.1006/jagm.2000.1143}}.

\bibitem{mitsche2015limiting}
Dieter Mitsche and Juanjo Ru{\'e}.
\newblock On the limiting distribution of the metric dimension for random
  forests.
\newblock {\em European Journal of Combinatorics}, 49:68--89, 2015.

\bibitem{MozesOW08}
Shay Mozes, Krzysztof Onak, and Oren Weimann.
\newblock Finding an optimal tree searching strategy in linear time.
\newblock In {\em {SODA}}, pages 1096--1105, 2008.

\bibitem{NesetrilM06}
Jaroslav Ne\v{s}et\v{r}il and Patrice~Ossona de~Mendez.
\newblock Tree-depth, subgraph coloring and homomorphism bounds.
\newblock {\em Eur. J. Comb.}, 27(6):1022--1041, 2006.
\newblock \href {https://doi.org/10.1016/j.ejc.2005.01.010}
  {\path{doi:10.1016/j.ejc.2005.01.010}}.

\bibitem{seager2014locating}
Suzanne Seager.
\newblock Locating a backtracking robber on a tree.
\newblock {\em Theoretical Computer Science}, 539:28--37, 2014.

\bibitem{slater1975leaves}
P.J. Slater.
\newblock Leaves of trees.
\newblock In {\em Proc. Sixth Southeastern Conf. Combin., Graph Theory,
  Computing}, number~14 in Congressus Numer., pages 549--559, 1975.

\end{thebibliography}

\end{document}